\documentclass{amsart}

\usepackage{color}
\newtheorem{theorem}{Theorem}[section]
\newtheorem{lemma}[theorem]{Lemma}

\theoremstyle{definition}

\theoremstyle{remark}

\numberwithin{equation}{section}

\definecolor{forestgreen}{cmyk}{0.91,0,0.88,0.12}
\definecolor{darkorange}{rgb}{1.00,0.55,0.00}
\definecolor{violet}{rgb}{0.5,0,0.8}

\newcommand{\R}{\mathbb R}

\begin{document}

\title[Singular solutions of elliptic equations]{Singular solutions of elliptic equations with iterated exponentials}

\author{Marius Ghergu}

\address{Institute of Mathematics “Simion Stoilow” of the Romanian Academy, P.O. Box 1-764, 014700 Bucharest, Romania}
\address{School of Mathematics and Statistics, University College Dublin, Belfield, Dublin 4, Ireland}

%Rouge, Louisiana 70803}
%\curraddr{Department of Mathematics and Statistics,
%Case Western Reserve University, Cleveland, Ohio 43403}
\email{marius.ghergu@ucd.ie}

\author{Olivier Goubet}
\address{Laboratoire Ami\'enois de Math\'ematiques Fondamentales et Appliqu\'ee, (LAMFA UMR 7352 CNRS UPJV)
33 rue saint-Leu, Université de Picardie Jules Verne
80039 Amiens
France}
\email{olivier.goubet@u-picardie.fr}

%    General info

%\date{January 1, 2001 and, in revised form, June 22, 2001.}



\begin{abstract}
We construct positive singular solutions for the problem $-\Delta u=\lambda \exp (e^u)$ in $B_1\subset \R^n$ ($n\geq 3$), $u=0$ on $\partial B_1$,  having a prescribed behaviour around the origin. Our study extends the one in Y. Miyamoto [Y. Miyamoto, A limit equation and bifurcation diagrams of semilinear elliptic equations with general supercritical growth. J. Differential Equations \textbf{264} (2018),  2684--2707] for such nonlinearities. Our approach is then carried out to elliptic equations featuring iterated exponentials.
\end{abstract}

\maketitle

\section{Introduction and the main results}
Consider the problem
\begin{equation}\label{main}
\left\{
\begin{aligned}
-\Delta u&=\lambda \exp({e^u})&&\quad\mbox{ in }B_1\setminus\{0\},\\
u&=0&&\quad\mbox{ on }\partial B_1,
\end{aligned}
\right.
\end{equation}
where $B_1\subset\R^n$ $(n\geq 3)$ is the open unit ball, $\lambda>0$ is a real number and
$$\exp({e^u})=e^{e^u}.
$$

The related problem, also known as the {\it Gelfand problem}, namely
\begin{equation}\label{gelfand}
\left\{
\begin{aligned}
-\Delta u&=\lambda {e^u} &&\quad\mbox{ in }B_1\setminus\{0\},\\
u&=0 &&\quad\mbox{ on }\partial B_1,
\end{aligned}
\right.
\end{equation}
has been long investigated starting with J. Liouville since 1853 (see  \cite{L1853}). One particular feature of \eqref{gelfand} is that for $\lambda=2(n-2)>0$ it has the explicit singular solution $u_s(x)=-2\ln|x|$.
Joseph and Lundgren \cite{JL1972} completely determined the structure of the radial solutions of \eqref{gelfand} emphasizing the role of the singular solution $u_s$ in the global picture of the solution set to \eqref{gelfand}. Thanks to the standard Hardy inequality, the  explicit singular solution $u_s$ is stable for all space dimensions $n\geq 10$. Further studies related to \eqref{gelfand} are contained in \cite{DD2007, GL2015, M2018b, MP1980, MP1988, P1990} and in the monograph \cite{Dbook}. Problems with exponential nonlinearities also appear in other contexts involving higher order operators \cite{AGGM2005, AGG2006, BGM2006, DGGW2013}, $p$-Laplace operators \cite{GPP1994} or $k$-Hessian operators \cite{J2004, JS2002} or even systems of coupled equations \cite{DG2014, G2019}.

Returning  to \eqref{main} we point out that such a problem does not possess an explicit singular solution. However, we are able to construct a radial singular solution $u^*$ with a prescribed behaviour around the origin. We prove:

\begin{theorem}\label{th1}
There exists a unique $\lambda^*>0$ such that \eqref{main} has a singular solution $u^*$ such that, as $|x|\to 0$ we have
\begin{equation}\label{asympt1}
u^*\Big(\frac{x}{\sqrt{\lambda^*}} \Big)=\ln\left[2\ln \frac{1}{|x|}+\ln\frac{n-2}{\ln\frac{1}{|x|}}+\ln\left(1+\frac{\ln \ln\frac{1}{|x|}}{2\ln\frac{1}{|x|}}\right)  \right]+O\Big(\ln^{-2}\frac{1}{|x|}\Big),
\end{equation}
and

\begin{equation}\label{gradi}
\frac{1}{\sqrt{\lambda^*}}\big|\nabla u^*\big|\Big(\frac{x}{\sqrt{\lambda^*}}\Big)=
\frac{1}{|x|\ln\frac{1}{|x|}}+\frac{\ln \Big(\ln\frac{1}{|x|}\Big)}{2|x| \ln^2 \frac{1}{|x|}}+O\Big( \frac{1}{|x|\ln^2\frac{1}{|x|}}\Big).
\end{equation}

\end{theorem}
Letting $\rho=\ln \frac{1}{|x|}$ and using the Maclaurin series approximation we may re-write \eqref{asympt1} as
\begin{equation}\label{asympt2}
u^*\Big(\frac{x}{\sqrt{\lambda^*}} \Big)=\ln(2\rho)+\frac{1}{2\rho}\ln\frac{n-2}{\rho}-\frac{1}{8\rho^2}\ln^2\frac{1}{\rho}+\frac{1}{4\rho^2}\ln \rho+O\Big(\frac{1}{\rho^2}\Big)
\end{equation}
as $\rho\to \infty$.

The related problem
\begin{equation}\label{kw}
\left\{
\begin{aligned}
-\Delta u&=\lambda \exp({u^p})&&\quad\mbox{ in }B_1\setminus\{0\},\\
u&=0&&\quad\mbox{ on }\partial B_1,
\end{aligned}
\right.
\end{equation}
was recently studied in \cite{KW2018}. It is proved in \cite{KW2018} that \eqref{kw} has a singular solution $(\lambda^*, u^*)$ that satisfies
$$
u^*\Big(\frac{x}{\sqrt{\lambda^*}} \Big)=\left(2\ln \frac{1}{|x|}-\big(1-\frac{1}{p}\big)\ln\ln\frac{1}{|x|} \right)^{1/p}+o\Big(\ln^{-1+\frac{1}{p}}\frac{1}{|x|}\Big)\mbox{ as }|x|\to 0.
$$
Also  the Morse index of $u^*$ is infinite (resp finite) provided $3\leq n\leq 9$ (resp. $n\geq 11$).

We would like to point out that in \cite{M2018} a positive radial singular solution $U$ of
\begin{equation}\label{UM}
\left\{
\begin{aligned}
-\Delta U&= \exp({e^U})&&\quad\mbox{ in }B_R\setminus\{0\},\  R>0,\\
U&=0&&\qquad\mbox{ on }\partial B_R,
\end{aligned}
\right.
\end{equation}
is constructed. Such a singular solution $U$ has the property that
\begin{equation}\label{F}
U(r)=F^{-1}\Big(\frac{r^2}{2(n-2)}(1+o(1))\Big)\quad\mbox{ as }r\to 0,
\end{equation}
where
\begin{equation}\label{FF}
F(t)=\int_t^\infty\exp(-e^s)ds.
\end{equation}
We are able to show that the solution $U(r)$ of \eqref{UM} coincides with $u^*(\sqrt{\lambda^*}r)$ in a neighbourhood of the origin. Thus, we may further investigate the bifurcation problem
\begin{equation}\label{bif}
\left\{
\begin{aligned}
-\Delta u&=\lambda \exp({e^u})&& \quad\mbox{ in }B_1,\\
 u&>0 &&\quad\mbox{ in }B_1,\\
u&=0&&\quad\mbox{ on }\partial B_1.
\end{aligned}
\right.
\end{equation}
By the classical result of Gidas, Ni and Nirenberg \cite{GNN1979} all solutions of \eqref{bif} are radially symmetric. Furthermore (see \cite{M2018}) the solution set of \eqref{bif} can be described as $\{(\lambda(\rho), u(\rho))\}$ where $\rho=\|u(\rho)\|_{L^\infty(B_1)}$ and $\lambda(0)=0$. Hence, the solution set $(\lambda,u)$ is a curve emanating from $(\lambda,u)=(0,0)$. Using \cite[Theorem 1.1, Corollary 1.2, Corollary 1.3]{M2018} we have:
\begin{theorem}\label{th2}
Let $u^*$ be the solution obtained in Theorem \ref{th1}.
\begin{enumerate}
\item[(i)] If $n\geq 11$, then the Morse index of $u^*$ is finite;
\item[(ii)] If $3\leq n\leq 9$, then the Morse index of $u^*$ is infinite. Furthermore:
\begin{enumerate}
\item[(ii1)] The curve $(\lambda(\rho), u(\rho))$ has infinitely many turning points around $\lambda^*$. In particular problem \eqref{bif} has infinitely many solutions for $\lambda=\lambda^*$;

\item[(ii2)] The number of intersection points between $u(\rho)$ and the singular solution $u^*$ tends to infinity as $\rho\to \infty$.
\end{enumerate}
\end{enumerate}
\end{theorem}

We also address in this article the similar problem
with iterared exponential, that reads for $m\geq 2$ and $G_0(y)=y$ and $G_m(y)=\exp(G_{m-1}(y))$

\begin{equation}\label{mainiterated}
\left\{
\begin{aligned}
-\Delta u&=\lambda \exp({G_m(u)})&&\quad\mbox{ in }B_1\setminus\{0\},\\
u&=0&&\quad\mbox{ on }\partial B_1.
\end{aligned}
\right.
\end{equation}

For problem \eqref{mainiterated} we prove

\begin{theorem}\label{th3}
Let $m\geq 2$ and let $H_m(y)=\ln(H_{m-1}(y))$ be the iterated logarithm ($H_0(y)=y$).
There exists a unique $\lambda^*>0$ such that \eqref{mainiterated} has a singular solution $u^*$ such that, as $|x|\to 0$ we have,
for $\rho=\ln \frac{1}{|x|}$,
\begin{equation}\label{asymptm}\begin{split}
u^*\Big(\frac{x}{\sqrt{\lambda^*}} \Big)=H_m(2\rho)+H'_m(2\rho)\left(\ln(2(n-2))-\sum_{j=1}^m H_j(2\rho) \right)+\\
 -\frac{1}{4\rho}H'_m(2\rho)(\ln \rho)^2+O\Big(\rho^2\Big),
\end{split}\end{equation}
and
\begin{equation}\label{gradim}
\frac{1}{\sqrt{\lambda^*}}\big|\nabla u^*\big|\Big(\frac{x}{\sqrt{\lambda^*}}\Big)=2\frac{H'_m(2\ln(\frac{1}{|x|}))}{|x|}+O\Big(\frac{1}{|x| \ln^2(\frac{1}{|x|})}\Big) .
\end{equation}
\end{theorem}

The next sections contain the proofs of the main results. Throughout this paper for any functions $f(t)$, $g(t)$ defined in a neighbourhood of infinity, we use the notation $f(t)=O(g(t))$ (resp. $f(t)=o(g(t))$ as $t\to \infty$ to indicate that $|\frac{f(t)}{g(t)}|$ is bounded (resp. tends to zero) as $t\to \infty$. A similar notation is used for $t\to 0$. Also, the symbols $C$, $c$ stand for generic positive constants whose values may be different on each occurence.

\section{Proof of Theorem \ref{th1}}

Let $u$ be a radial solution of \eqref{main}. Letting $v(x)=u\big(\frac{x}{\sqrt{\lambda}} \big)$ we find
\begin{equation}\label{v}
\left\{
\begin{aligned}
&v_{rr}+\frac{n-1}{r}v_r+\exp(e^v)=0\quad\mbox{ for all }0<r<\sqrt{\lambda},\\
&v(\sqrt{\lambda})=0.
\end{aligned}
\right.
\end{equation}
Letting $t=-\ln r$ and $w(t)=v(r)$ we find that $w\in C^2(-\ln \sqrt{\lambda}, \infty)$ satisfies
\begin{equation}\label{w}
\left\{
\begin{aligned}
&w_{tt}-(n-2) w_t+\exp(-2t+e^w)=0, \;w>0\quad\mbox{ for all }\, -\ln \sqrt{\lambda}<t<\infty,\\
&w(-\ln \sqrt{\lambda})=0.
\end{aligned}
\right.
\end{equation}
We next look for a solution $w(t)$ of \eqref{w} in the form
\begin{equation}\label{wsol}
w(t)=\ln\big(2t+\varphi(t)\big)+\eta(t)
\end{equation}
where
$$
\varphi(t)=\ln \frac{n-2}{t}+\ln\Big(1+\frac{\ln t}{2t}\Big).
$$
Let us observe that, as $t\to \infty$, we have
\begin{equation}\label{vfi}
\begin{aligned}
\varphi_t(t)&=-\frac{1}{t}+\frac{1-\ln t}{t(2t+\ln t)}=O\Big(\frac{1}{t}\Big),\\
\varphi_{tt}(t)&=\frac{(4t+1+\ln t)(2t+2\ln t-1)-2(t+1)(2t+\ln t)}{t^2(2t+\ln t)^2}=O\Big(\frac{1}{t^2}\Big).
\end{aligned}
\end{equation}
Letting $f=\ln \big(2t+\varphi(t)\big)$ we have
\begin{align}
f_t&=\frac{2+\varphi_t}{2t+\varphi}= O\Big(\frac{1}{t}\Big) \quad\mbox{ as }t\to \infty,\label{f1}\\
f_{tt}&=\frac{\varphi_{tt}}{2t+\varphi}-\Big(\frac{2+\varphi_t}{2t+\varphi}\Big)^2=O\Big(\frac{1}{t^2}\Big) \quad\mbox{ as }t\to \infty.\label{f2}
\end{align}

Also,
$$
\begin{aligned}
e^\varphi-(n-2)f_t&=\frac{n-2}{t}+\frac{(n-2)\ln t}{2t^2}-(n-2)f_t\\
&=\frac{n-2}{t}\cdot \frac{2t(\varphi+\ln t)-2t^2 \varphi_t+\varphi \ln t}{2t(2t+\varphi)}.
\end{aligned}
$$
Observe that
$$
\varphi(t)+\ln t=O(1)\quad\mbox{ as } t\to \infty,
$$
so by \eqref{vfi} and the above calculations we find
\begin{equation}\label{ff}
e^\varphi-(n-2)f_t=O\Big(\frac{1}{t^2}\Big) \quad\mbox{ as }t\to \infty.
\end{equation}
Using equation \eqref{wsol} we have
$$
\begin{aligned}
\exp(-2t+e^w)=&\exp\big[ e^{\eta}(2t+\varphi)-2t \big] \\
=&\exp\big[ (e^{\eta}-1)(2t+\varphi)+\varphi \big] \\
=&e^\varphi\left\{\exp\big[ (e^{\eta}-1)(2t+ \varphi)\big]-(e^{\eta}-1)(2t+\varphi)-1\right\}\\
&+e^\varphi (e^{\eta}-1)(2t+\varphi)+e^\varphi \\
=&e^\varphi\left\{\exp\big[ (e^{\eta}-1)(2t+\varphi)\big]-(e^{\eta}-1)(2t+\varphi)-1\right\} \\
&+e^\varphi (e^{\eta}-\eta-1)(2t+ \varphi)+e^\varphi \eta (2t+\varphi)+e^\varphi,
\end{aligned}
$$
so
\begin{equation}\label{e}
\exp(-2t+e^w)=F_1(t) \eta+F_2(t, \eta)+F_3(t, \eta)+2(n-2)\eta+e^\varphi,
\end{equation}
where
\begin{align}
F_1(t) \eta&=e^\varphi \eta(2t+\varphi)-2(n-2)\eta=\Big(\frac{(n-2)\ln t}{t}+e^\varphi \varphi\Big)\eta\label{ff1}\\
F_2(t, \eta)&=e^\varphi (e^{\eta}-\eta-1)(2t+\varphi)\label{ff2}\\
F_3(t, \eta)&=e^\varphi\left\{\exp\big[ (e^{\eta}-1)(2t+\varphi)\big]-(e^{\eta}-1)(2t+ \varphi)-1\right\}.\label{ff3}
\end{align}
Let also $F_0(t)=e^\varphi +f_{tt}-(n-2)f_t$.

Using the first equation of \eqref{w} together with \eqref{wsol}, \eqref{e}-\eqref{ff3} we deduce that $\eta$ satisfies
\begin{equation}\label{eta1}
\eta_{tt}-(n-2) \eta_t+2(n-2)\eta+F(t, \eta)=0, \quad\mbox{ for all }\, -\ln \sqrt{\lambda}<t<\infty,
\end{equation}
where
$$
F(t, \eta)=F_1(t) \eta+F_2(t, \eta)+F_3(t, \eta)+F_0(t).
$$
We shall show that equation \eqref{eta1} has a solution $\eta\in X_T$ where $T>0$ is a real number and
\begin{equation}\label{eqxt}
X_T=\{\eta\in C(T, \infty), \; \eta=O(1/t^2)\mbox{ as }t\to \infty\},
\end{equation}
equipped with the norm $\|\eta\|_{X_T}=\sup_{t>T} t^2|\eta(t)|$.
As in \cite{KW2018}, we discuss in the following the case $3\leq n\leq 9$, the case $n\geq 10$ being similar. We transform \eqref{eta1} into the fixed point equation
\begin{equation}\label{fixedp}
\eta=\Psi[\eta],
\end{equation}
where the integral operator $\Psi[\eta]$ is given by
\begin{equation}\label{integraloperator}
\Psi[\eta](t)=-\frac{1}{\mu} e^{\frac{(n-2)t}{2}} \int_t^\infty e^{-\frac{(n-2)s}{2}}\sin(\mu(s-t))F(s,\eta) ds,
\end{equation}
and $\mu=\sqrt{(n-2)(10-n)}\geq 1$. The existence of a solution to \eqref{fixedp} will be derived by means of the contraction principle; to this end, for $M>0$ set
\begin{equation}\label{eqst}
\Sigma_M=\{\eta\in X_T:\|\eta\|_{X_T}\leq M\}.
\end{equation}

\begin{lemma}\label{contraction}
There exist $M,T>0$ such that $\Psi(\Sigma_M)\subset \Sigma_M$ and $\Psi:\Sigma_M\to \Sigma_M$ is a contraction.
\end{lemma}
\begin{proof}
Using \eqref{f2} and \eqref{ff} we have
$$
|F_0(t)|\leq \frac{A}{t^2}\quad\mbox{ for  $t>0$ large, }
$$
where $A>0$. Let now $\eta\in \Sigma_M$. Then, for  $t>0$ large, we estimate
$$
\begin{aligned}
|F_1(t)\eta|\leq & \frac{C\ln t}{t}|\eta(t)|\leq \frac{CM\ln t}{t^3},\\
|F_2(t, \eta)|\leq & C(e^\eta-\eta-1)\leq C\eta^2(t)\leq \frac{CM^2}{t^4}, \\
|F_3(t, \eta)|\leq & \frac{C}{t}(e^\eta-1)^2(2t+\varphi)^2\leq \frac{C}{t}\eta^2(t) t^2 \leq \frac{CM^2}{t^3}.
\end{aligned}
$$
Thus, by taking $M=2A$ and $T>0$ large enough we have
$$
|F(t, \eta)|\leq \frac{M}{t^2}\quad\mbox{ for all }t>T.
$$
Using this fact we have
$$
\begin{aligned}
\Big|\Psi[\eta](t) \Big|& \leq \frac{M}{\mu}
e^{\frac{(n-2)t}{2}} \int_t^\infty \frac{e^{-\frac{(n-2)s}{2}}}{s^2} ds\\
&\leq \frac{M}{\mu}
\frac{e^{-\frac{(n-2)t}{2}}}{t^2}  \int_t^\infty e^{-\frac{(n-2)s}{2}}  ds\\
&=\frac{M}{\mu}\cdot \frac{2}{n-2}\leq M,
\end{aligned}
$$
since $(n-2)\mu \geq 2$. This shows that $\Psi(\Sigma_M)\subset \Sigma_M$.

To prove that $\Psi:\Sigma_M\to \Sigma_M$ is a contraction, let $\eta_1, \eta_2\in \Sigma_M$. Then
\begin{equation}\label{cont1}
\big|\Psi[\eta_1](t)-\Psi[\eta_2](t)\big|\leq \frac{1}{\mu} \sum_{k=1}^3
e^{\frac{(n-2)t}{2}} \int_t^\infty e^{-\frac{(n-2)s}{2}}\big|F_k(s,\eta_1)-F_k(s, \eta_2)\big| ds.
\end{equation}
From \eqref{ff3} we have
$$
\big|F_1(s)(\eta_1(s)-\eta_2(s))\big|\leq
 \frac{C\ln s}{s}|\eta_1(s)-\eta_2(s)|\leq \frac{C\ln s}{s^3}\|\eta_1-\eta_2\|_{X_T}.
$$
Using the Mean Value Theorem and \eqref{ff2} we estimate
$$
\begin{aligned}
\big|F_2(s,\eta_1(s))-F_2(s, \eta_2(s))\big|& \leq C\big|(e^{\eta_1(s)}-\eta_1(s))-(e^{\eta_2(s)}-\eta_2(s))\big|\\
&\leq C\big| e^{\eta(s)}-1\big| |\eta_1(s)-\eta_2(s)|,
\end{aligned}
$$
for some $\eta(s)$ between $\eta_1(s)$ and $\eta_2(s)$. Thus,
$$
\big|F_2(s,\eta_1(s))-F_2(s, \eta_2(s))\big|\leq C|\eta(s)|  |\eta_1(s)-\eta_2(s)|
\leq \frac{C\ln s}{s^3}\|\eta_1-\eta_2\|_{X_T}.
$$
Also, leting $\theta_j(s)=(e^{\eta_j(s)}-1)(2s+\eta_j(s))$, by \eqref{ff1} and the Mean Value Theorem we estimate
$$
\begin{aligned}
\big|F_3(s,\eta_1)-F_3(s, \eta_2)\big|& = e^{\varphi(s)}\big|(e^{\theta_1(s)}-\theta_1(s))-(e^{\theta_2(s)}-\theta_2(s))\big|\\
&\leq \frac{C}{s}\big| e^{\theta(s)}-1\big| |\theta_1(s)-\theta_2(s)|\\
& \leq \frac{C}{s}\big|\theta(s)\big| |\theta_1(s)-\theta_2(s)|,
\end{aligned}
$$
for some $\theta(s)$ between $\theta_1(s)$ and $\theta_2(s)$. By the Mean Value Theorem we further estimate
$$
\begin{aligned}
\big|\theta_1(s)-\theta_2(s)\big|&=\big|(\eta_1(s)-\eta_2(s))(e^{\eta_1(s)}-1)+(\eta_2(s)+2s)(e^{\eta_1(s)}-e^{\eta_2(s)}) \big|\\
& =\big| e^{\eta_1(s)}-1+(\eta_2(s)+2s)(e^{\eta(s)}\big| \big|\eta_1(s)-\eta_2(s)|\\
&\leq C \big|\eta_1(s)-\eta_2(s)|,
\end{aligned}
$$
where $\eta(s)$ lies between $\eta_1(s)$ and $\eta_2(s)$. Hence,
$$
\big|F_3(s,\eta_1)-F_3(s, \eta_2)\big|
\leq\frac{C}{s}|\eta_1(s)-\eta_2(s)|\leq \frac{C}{s^3}\|\eta_1-\eta_2\|_{X_T}.
$$

Now, using the above estimates in \eqref{cont1} we deduce
$$
\big|\Psi[\eta_1](t)-\Psi[\eta_2](t)\big|\leq\frac{C}{t^2}\|\eta_1-\eta_2\|_{X_T}.
$$
By taking now $T>0$ large enough it follows that $\Psi:\Sigma_M\to \Sigma_M$ is a contraction.
\end{proof}

We are now in a position to prove the result in Theorem \ref{th1}. First, there exists $\eta\in \Sigma_M$ a a solution of \eqref{fixedp}, that is, $\eta\in C^2(T, \infty)$ satisfies
\begin{equation}\label{eta01}
\left\{
\begin{aligned}
&\eta_{tt}-(n-2) \eta_t+2(n-2)\eta+F(t, \eta)=0, \quad\mbox{ for all }\, t>T,\\
&\eta(t)=O\big(1/t^2\big)\quad\mbox{ as }t\to \infty.
\end{aligned}
\right.
\end{equation}
Thus, the function $w$ given by \eqref{wsol} is positive in a neighbourhood of infinity and satisfies
\begin{equation}\label{wsol1}
w_{tt}-(n-2) w_t+\exp(-2t+e^w)=0 \quad\mbox{ for all }\,T<t<\infty.
\end{equation}
We claim that $w$ vanish at some point $T_0>T$. Otherwise, $w>0$ in $(T, \infty)$ and by the continuation principle $w$ satisfies \eqref{wsol1} on the whole real line. We claim that this implies that $w$ is monotone increasing. Indeed, assuming the contrary, there would exist a minimum point $t_0\in \R$ at which $w_t(t_0)=0$ and $w_{tt}(t_0)\geq 0$ which contradicts \eqref{wsol1}. Hence, $w$ is monotone increasing and there exists $L:=\lim_{t\to -\infty}w(t)\in [0, \infty)$.

Multiply in \eqref{wsol1} by $e^{-(n-2)t}$ and using the fact that $w$ in increasing, we find
$$
\frac{d}{dt}\Big(e^{-(n-2)t}w_t\Big)=-\exp(-nt+e^w)\leq -\exp(-nt+e^L).
$$
Integrating in the above equality over the interval $[t,t_0]$, $-\infty<t<t_0$, we find
$$
e^{-(n-2)t_0}w_t(t_0)-e^{-(n-2)t}w_t(t)\leq C\big(e^{-nt_0}-e^{-nt}\big)\quad\mbox{ for all }-\infty<t<t_0,
$$
where $C=\frac{1}{n}\exp(e^L)>0$. This implies further that $\lim_{t\to -\infty}w_t(t)=\infty$ which contradicts the fact that $\lim_{t\to-\infty}w(t)$ is finite.
This shows that $w$ vanishes at some point $T^*\in \R$ and $w\in C^2(T^*, \infty)$ satisfies
\begin{equation}\label{eta2}
\left\{
\begin{aligned}
&w_{tt}-(n-2) w+\exp(-2t+e^w)=0,\; w>0 \quad\mbox{ for all }\, t>T,\\
&w(t)=\ln(2t+\varphi(t))+O\Big(\frac{1}{t^2}\Big)\quad\mbox{ as }t\to \infty,\\
&w(T^*)=0.
\end{aligned}
\right.
\end{equation}
Letting $\lambda^*=e^{-2T^*}$ and $u^*(x)=v^*\big(\sqrt{\lambda^*}x \big)$ (where $v$ is the solution of \eqref{v} with $\lambda=\lambda^*$) we obtain that $u^*$ is a solution of \eqref{main} with $\lambda=\lambda^*$  which satisfies \eqref{asympt1}. Concerning the proof of the  asymptotic behaviour in \eqref{asympt2} we have
\begin{equation}\label{asympt3}
w(t)=\ln(2t)+\ln\Big(1+\frac{\varphi(t)}{2t}\Big)+O\Big(\frac{1}{t^2}\Big)=\ln(2t)+\frac{\varphi(t)}{2t}-\frac{\varphi^2(t)}{8t^2}+O\Big(\frac{1}{t^2}\Big)
\end{equation}
as $t\to \infty$. Since
$$
\begin{aligned}
\varphi(t)&=\ln\frac{n-2}{t}+\ln\Big(1+\frac{\ln t}{2t}\Big)\\
&=\ln\frac{n-2}{t}+\frac{\ln t}{2t}-\frac{\ln^2 t}{8t^2}+O\Big(\frac{1}{t^2}\Big)\quad\mbox{ as }t\to \infty,
\end{aligned}
$$
we have
\begin{align}
\frac{\varphi(t)}{2t}&= \frac{1}{2t}\ln \frac{n-2}{t}+\frac{1}{4t^2}\ln t+O\Big(\frac{1}{t^2}\Big)\quad\mbox{ as }t\to \infty,\label{asympt4}\\
\frac{\varphi^2(t)}{8t^2}&=\frac{1}{8t^2}\ln^2 \frac{1}{t}+O\Big(\frac{1}{t^2}\Big)\quad\mbox{ as }t\to \infty.
 \label{asympt5}
\end{align}
Using \eqref{asympt4}-\eqref{asympt5} in \eqref{asympt3} we obtain
$$
w(t)=\frac{1}{2t}\ln\frac{n-2}{t}-\frac{1}{8t^2}\ln^2\frac{1}{t}+\frac{1}{4t^2}\ln t+O\Big(\frac{1}{t^2}\Big)
\quad\mbox{ as } t \to \infty,
$$
which proves \eqref{asympt2}.

We next focus on the expansion of the gradient $\nabla u^*$ around the origin.

\begin{lemma}\label{ett} The solution $\eta$ of \eqref{eta01} satisfies
\begin{equation}\label{etat}
\eta_t(t)=O\Big(\frac{1}{t^2}\Big)\quad\mbox{ as }t\to \infty.
\end{equation}
\end{lemma}
\begin{proof} From the proof of Lemma \ref{contraction} we infer that
$$
\eta_{tt}-(n-2)\eta_t=g(t,\eta)=O\Big(\frac{1}{t^2}\Big)\quad\mbox{ as }t\to \infty.
$$
\noindent Then $\eta_t$ is solution to the integral equation
$$
\eta_t(t)=-\int_t^{+\infty} e^{(n-2)(t-s)}g(s,\eta(s))ds.
$$
\noindent The result follows promptly.
\end{proof}

Now, from \eqref{wsol}, \eqref{ff}  and \eqref{etat} we have
\begin{equation}\label{wt}
w_t=f_t+\eta_t=\frac{1}{n-2}e^{\varphi}+O\Big( \frac{1}{t^2}\Big)=
\frac{1}{t}+\frac{\ln t}{2t^2}+O\Big( \frac{1}{t^2}\Big)  \quad\mbox{ as }t\to \infty.
\end{equation}

Recall that the singular solution $u^*$ is given by $u^*\big(\frac{r}{\sqrt{\lambda^*}}\big)=v(r)=w(t)$ where $v$ and $w$ are solutions of \eqref{v} and \eqref{w} (with $\lambda=\lambda^*$).  Since  $|w_t|=|v_r| | \frac{dr}{dt}|=r|v_r|$, from \eqref{wt} we find

$$
\frac{1}{\sqrt{\lambda^*}}\big|\nabla u^*\big|\Big(\frac{x}{\sqrt{\lambda^*}}\Big)=
\frac{1}{|x|\ln\frac{1}{|x|}}+\frac{\ln \Big(\ln\frac{1}{|x|}\Big)}{2|x| \ln^2 \frac{1}{|x|}}+O\Big( \frac{1}{|x|\ln^2\frac{1}{|x|}}\Big) \quad\mbox{ as }|x|\to 0.
$$

\section{Proof of Theorem \ref{th2}}

Let us recall that the solution $U$ of \eqref{UM} is obtained in \cite[Lemma 3.1]{M2018} as
\begin{equation}\label{Ur}
U(r)=F^{-1}\Big(\frac{1+x(t)}{2(n-2) e^{2t}} \Big) \quad\mbox{ where $t=-\ln r$.}
\end{equation}
Further, the unknowns $x(t)$ and $y(t)=\frac{dx}{dt}(t)$ are found as a unique fixed point through a contraction mapping in the set
$$
B_\varepsilon:=\Big\{(x,y)\in C([T, \infty),\R^2): \|x\|_{L^\infty[T,\infty)}+ \|y\|_{L^\infty[T,\infty)}<\varepsilon\Big\}.
$$
From \eqref{Ur} we find
\begin{equation}\label{sysF}
\left\{
\begin{aligned}
&x(t)=2(n-2)e^{2t} F(U(r))-1,\\
&y(t)=\frac{dx}{dt}(t),
\end{aligned}
\quad \mbox{ where } t=-\ln r>0\mbox{ is large}.
\right.
\end{equation}
Let now $u^*$ be the singular solution of \eqref{main} for $\lambda^*$ constructed in Theorem \ref{th1}. Let $v^*(x)=u^*\big(\frac{x}{\sqrt{\lambda^*}} \big)$ which satisfies \eqref{v}, that is the same equation as \eqref{UM} in $B_{\sqrt{\lambda^*}}$.
Define
\begin{equation}\label{sysF2}
\left\{
\begin{aligned}
&x^*(t)=2(n-2)e^{2t} F(v^*(r))-1,\\
&y^*(t)=\frac{dx^*}{dt}(t),
\end{aligned}
\quad \mbox{ where } t=-\ln r>0\mbox{ is large}.
\right.
\end{equation}
In order to prove that $v^*\equiv U$ in a neighbourhood of infinity it is enough to show that $(x^*(t), y^*(t))$ belongs to $B_\varepsilon$.

Letting $w^*(t)=v^*(r)$ where $t=-\ln r$, we have that $w^*$ satisfies \eqref{w}. Thus, \eqref{sysF2} reads
\begin{equation}\label{sysF3}
\left\{
\begin{aligned}
&x^*(t)=2(n-2)e^{2t} F(w^*(t))-1,\\
&y^*(t)=\frac{dx^*}{dt}(t),
\end{aligned}
\quad \mbox{ where } t=-\ln r>0\mbox{ is large}.
\right.
\end{equation}
In order to conclude the proof it suffices to show that
\begin{equation}\label{limit}
\lim_{t\to \infty}x^*(t)=\lim_{t\to \infty}y^*(t)=0.
\end{equation}
Then, for large $T>0$ we have $(x^*,y^*)\in B_\varepsilon$ so, by the uniqueness of the fixed point $(x^*,y^*)=(x,y)$ in a neighbourhood of infinity, that is $v^*\equiv U$ in a neighbourhood of the origin and then we conclude from Theorem 1.1, Corollary 1.2, Corollary 1.3 from \cite{M2018}. Let us now turn to \eqref{limit}.

From the definition of $F$ in \eqref{F} and L'Hospital's rule we find
\begin{equation}\label{FFlim}
\lim_{t\to \pm \infty}\frac{F(t)}{\exp(-|t|-e^t)}=1.
\end{equation}
Thus,
$$
\begin{aligned}
\lim_{t\to \infty}x^*(t)&=2(n-2)\lim_{t\to \infty} e^{2t}F(w^*(t))-1\\
&=2(n-2)\lim_{t\to \infty} \frac{1}{e^{-2t} \exp(w^*(t)+e^{w^*(t)})}-1.
\end{aligned}
$$
Now, using \eqref{wsol} we find
$$
\begin{aligned}
e^{-2t} \exp(w^*(t)+e^{w^*(t)})&=\exp\Big(-2t+\ln(2t+\varphi(t))+\eta+(2t+\varphi)e^\eta\Big)\\
&=(2t+\varphi(t))e^{ \varphi(t)} \exp\Big(\eta+(e^\eta-1)(2t+\varphi)\Big)\\
&=(2t+\varphi(t))\frac{n-2}{t}\Big(1+\frac{\ln t}{2t}\Big) \exp\Big(\eta+(e^\eta-1)(2t+\varphi)\Big).
\end{aligned}
$$
Since $\eta(t)=O(t^{-2})$ as $t\to \infty$ it follows that
$$
\lim_{t\to \infty} e^{-2t} \exp(w^*(t)+e^{w^*(t)})=2(n-2),
$$
which proves the first part of \eqref{limit}. For the second part, we first note that
$$
y^*(t)=\frac{dx^*}{dt}(t)
=4(n-2)e^{2t}F(w^*(t))-2(n-2)\frac{e^{2t}w_t^*(t)}{\exp(e^{w^*(t)})}.
$$
Thus, from the above arguments we find
\begin{equation}\label{lim1}
\lim_{t\to \infty}y^*(t)=2-2(n-2) \lim_{t\to \infty} \frac{tw_t^*(t)}{te^{-2t}\exp(e^{w^*(t)})}.
\end{equation}
Using \eqref{wsol} and Lemma \ref{ett} we have
\begin{equation}\label{lim2}
\lim_{t\to \infty} tw^*_t(t)=1.
\end{equation}
We also have
$$
\begin{aligned}
te^{-2t}\exp(e^{w^*(t)})&=t\exp\Big(-2t+(2t+\varphi(t))e^{\eta(t)}\Big)\\
&=te^{\varphi(t)}\exp\Big((e^\eta-1)(2t+\varphi(t))\Big).
\end{aligned}
$$
Since $\eta(t)=O(t^{-2})$ as $t\to \infty$, we find
\begin{equation}\label{lim3}
\lim_{t\to \infty} te^{-2t}\exp(e^{w^*(t)})=n-2.
\end{equation}
Combining \eqref{lim1}, \eqref{lim2} and \eqref{lim3} we deduce $\lim_{t\to \infty}y^*(t)=0$ which finishes our proof. \qed

\section{Proof of Theorem \ref{th3}}

We follow the lines of the proof of Theorem \ref{th1}. For $m\geq 2$,
we plan to solve here the equation
\begin{equation}\label{wm}
\left\{
\begin{aligned}
&w_{tt}-(n-2) w_t+\exp(-2t+G_m(w))=0, \;w>0\quad\mbox{ for all }\, -\ln \sqrt{\lambda}<t<\infty,\\
&w(-\ln \sqrt{\lambda})=0.
\end{aligned}
\right.
\end{equation}
We look for a solution $w(t)$ of \eqref{wm} in the form
\begin{equation}\label{wsolm}
w(t)=H_m\big(2t+\varphi(t)\big)+\eta(t),
\end{equation}
where
$$
\varphi(t)=\ln \big( 2(n-2)H'_m(2t) \big).
$$
Here $H_m$ is the iterated logarithm function defined by $H_m(G_m(y))=y$.
We then have to solve the equation
\begin{equation}\label{etam}
\eta_{tt}-(n-2) \eta_t+2(n-2)\eta+F(t, \eta)=0, \quad\mbox{ for all }\, -\ln \sqrt{\lambda}<t<\infty,
\end{equation}
where
$$
F(t, \eta)=F_1(t) \eta+F_2(t, \eta)+F_0(t),
$$
\noindent with

\begin{align}
F_0(t)&=2(n-2)(H'_m(2t)-H_m'(2t+\varphi))-(n-2)\varphi_tH'_m(2t+\varphi)+
\big(H_m(2t+\varphi)\big)_{tt},\label{ff1m}\\
F_1(t) \eta&=\big( e^\varphi G'_m(H_m(2t+\varphi))-2(n-2)\big)\eta,\label{ff2m}\\
F_2(t, \eta)&=e^\varphi \left( \exp\big(G'_m(H_m(2t+\varphi)\eta\big)+\rho(\eta))-1-G'_m(H_m(2t+\varphi))\eta\right).\label{ff3m}
\end{align}
\noindent Here $\rho(\eta)$ is defined thanks to Taylor's formula as
\begin{equation}\label{ro}
G_m(w)=G_m(H_m(2t+\varphi)+\eta)=(2t+\varphi)+G'_m(H_m(2t+\varphi))\eta+\rho(\eta).
\end{equation}

The next result provides the estimates we need to construct our solution $\eta$ in the space  $\Sigma_M$ defined in \eqref{eqst}.

\begin{lemma}\label{lestimate}
Let $m\geq 2$.
\begin{enumerate}
\item[(a)] For $k=1,2,3$ we have
\begin{equation}\label{estq1}
H_m^{(k)}(t)=O\left(\frac{1}{t^k\ln t}  \right)
\end{equation}
and
\begin{equation}\label{estq2}
\frac{H''_m(t)}{H'_m(t)}=-\frac{1}{t}+O\left(\frac{1}{t\ln t}\right).
\end{equation}
\item[(b)] Also
\begin{equation}\label{estq3}
\varphi(t)\simeq -\ln t,\;\;  \varphi_t(t)=-\frac{1}{t}+O\left( \frac{1}{t\ln t}\right), \;\; \varphi_{tt}(t)=O\left( \frac{1}{t^2}\right).
\end{equation}
\item[(c)] For $k=1,2,3$ there holds
\begin{equation}\label{estq4}
G_m^{(k)}\big(H_m(t)\big)=O\left(t(\ln t)^{k+1}  \right).
\end{equation}
\end{enumerate}
\end{lemma}
\begin{proof}
(a) The estimates \eqref{estq1}-\eqref{estq2} follow from  the identities
\begin{align}
H'_m(t)&= \prod_{j=0}^{m-1} \frac{1}{H_j(t)},\label{estq11}\\
H''_m(t)&= -H'_m(t)\sum_{j=0}^{m-1}\frac{H'_j(t)}{H_j(t)},\label{estq12}\\
H'''_m(t) &=-H''_m(t)\sum_{j=0}^{m-1}\frac{H'_j(t)}{H_j(t)}+H'_m(t)\sum_{j=0}^{m-1}\left\{ \Big[\frac{H'_j(t)}{H_j(t)}\Big]^2-\frac{H''_j(t)}{H_j(t)}\right\}.\label{estq13}
\end{align}

(b) We have
\begin{equation}\label{fifi}
\varphi(t)=\ln(2(n-2))-\sum_{j=0}^{m-1}\ln(H_j(2t))\simeq -\ln t.
\end{equation}
From \eqref{estq2} we find
$$
\varphi_t(t)=2\frac{H''_m(2t)}{H'_m(2t)}=-\frac{1}{t}+O\left( \frac{1}{t\ln t}\right).
$$
Combining the above equality with \eqref{estq2} we find
$$
\begin{aligned}
\varphi_{tt}(t)&=-2\frac{d}{dt}\left\{ \sum_{j=0}^{m-1}\frac{H'_j(2t)}{H_j(2t)}  \right\}\\
&=4 \sum_{j=0}^{m-1}\left\{ \Big[\frac{H'_j(2t)}{H_j(2t)}\Big]^2-\frac{H''_j(2t)}{H_j(2t)}\right\}\\
&=O\left( \frac{1}{t^2}\right).
\end{aligned}
$$
(c) The estimate \eqref{estq4} follows from  the following computations
\begin{align*}
0\leq G'_m(t)&= G_m(t) \prod_{j=0}^{m-1} G_j(t),\\
0\leq G''_m(t)&\leq  m G_m(t) \Big(\prod_{j=0}^{m-1} G_j(t)\Big)^2,\\
0\leq G'''_m(t)&\leq  m^2 G_m(t) \Big(\prod_{j=0}^{m-1} G_j(t)\Big)^3.
\end{align*}
\end{proof}

\noindent We seek a solution $\eta\in \Sigma_M$ by Lemma \ref{contraction} as in the proof of Theorem \ref{th1}.

We begin with an upper bound for $F_1(t)$ defined in \eqref{ff2m}.
By Mean Value theorem we have that
\begin{equation}\label{og1}
\begin{aligned}
|F_1(t)|&=2(n-2) \frac{|H'_m(2t)-H'_m(2t+\varphi)|}{H'_m(2t+\varphi)}\\
&=2(n-2)\frac{|H''_m(s)|}{H'_m(2t)}\;|\varphi(t)|,
\end{aligned}
\end{equation}
for some $s$ between $2t+\varphi$ and $2t$. Since $s\mapsto |H''_m(s)|$ is decreasing and $\varphi(t)$ is negative for large $t>0$, from \eqref{og1} and \eqref{estq2}-\eqref{estq3} we find
\begin{equation}\label{og1.1}
\begin{aligned}
|F_1(t)|&=2(n-2) \frac{|H''_m(2t)|}{H'_m(2t)}\;|\varphi(t)|\leq c_{n,m} \frac{\ln t}{t}.
\end{aligned}
\end{equation}

We now bound $F_0(t)$ defined in \eqref{ff1m}. First, by the estimates \eqref{estq1} we have
\begin{equation}\label{estw1}
\left |H'_m(2t+\varphi)\varphi_t)+\Big((H_m(2t+\varphi)\Big)_{tt} \right| \leq c_m t^{-2}.
\end{equation}
Also, by the Mean Value Theorem and \eqref{estq1} we have
\begin{equation}\label{estw2}
\big| H'_m(2t)-H_m'(2t+\varphi)\big|\leq c|\varphi| |H''_m(s)|\leq \frac{c}{t^2},
\end{equation}
for some $s$ between $2t+\varphi(t)$ and $\varphi(t)$.
Hence, from \eqref{estw1}-\eqref{estw2} we find
\begin{equation}
|F_0(t)|\leq \frac{c}{t^2}.
\end{equation}

We now handle the nonlinear term $F_2(t,\eta)$ defined in \eqref{ff3m}. Observe first that $\rho(\eta)$ given in \eqref{ff3m}-\eqref{ro} can be written as
\begin{equation}\label{ro0}
\rho(\eta)=\eta^2 \int_0^1 G_m''(H_m(2t+\varphi)+ z\eta)(1-z)dz .
\end{equation}
Since $\varphi(r)\simeq -\ln t$ and $\eta\in \Sigma_M$, a further adjustment of $M$ and Mean Value Theorem together with \eqref{estq1} yields
$$
H_m(2t+\varphi)+ z\eta\leq H_m(2t)\quad\mbox{ for all }z\in [0,1].
$$
Since $G_m''$ is increasing, using \eqref{estq4} we deduce
\begin{equation}\label{roest1}
|\rho(\eta(t))|\leq \eta^2(t)G''_m(H_m(2t+\varphi(t))\leq \frac{c(\ln t)^3}{t^3}\quad\mbox{ for any } \eta\in \Sigma_M.
\end{equation}
Let now $\eta_1,\eta_2\in \sigma_M$ and denote $\theta_j(t)=\frac{\eta_j(t)}{H'_m(2t+\varphi(t)) }$. Using that
\begin{equation}\label{h}
\frac{1}{ c_m } \frac{1}{t (\ln t)^2}\leq H'_m(2t+\varphi)\leq c_m \frac{1}{t \ln t},
\end{equation}
we have
\begin{equation}\label{teta1}
|\theta_j(t)|\leq \frac{c(\ln t)^2}{t}.
\end{equation}
Since $G_m$ and $H_m$ are inverse each other, for $j=1,2$ we may write
$$
F_2(t,\eta_j)=e^\varphi\left(e^{\theta_j}-1-\theta_j\right)+e^{\varphi +\theta_j} (e^{\rho(\eta_j)}-1).
$$
Thus
\begin{equation}\label{teta2}
|F_2(t,\eta_2)-F_2(t,\eta_1)|\leq
e^\varphi\left|(e^{\theta_2}-1-\theta_2)-(e^{\theta_1}-1-\theta_1)\right|
+ e^{\varphi}\left| e^{\theta_2} (e^{\rho(\eta_2)}-1)-e^{\theta_1} (e^{\rho(\eta_1)}-1)   \right|.
\end{equation}
By the Mean Value Theorem, \eqref{h} and \eqref{teta1}, for some $\theta(t)$ between $\theta_1(t)$ and $\theta_2(t)$, we have
\begin{equation}\label{teta3}
\begin{aligned}
e^\varphi\left|(e^{\theta_2}-1-\theta_2)-(e^{\theta_1}-1-\theta_1)\right|&
\leq  e^\varphi |e^\theta-1||\theta_2-\theta_1|\\
&\leq \frac{C|\theta|}{t H_m'(2t+\varphi)} |\eta_2-\eta_1|\\
&\leq \frac{c(\ln t)^4}{t} |\eta_2-\eta_1|.
\end{aligned}
\end{equation}
Now,  the second term in \eqref{teta2} is bounded from above by
\begin{equation}\label{teta4}
e^{\varphi +\theta_2} |(e^{\rho(\eta_2)}-e^{\rho(\eta_1)}|+e^\varphi|(e^{\rho(\eta_1)}-1)(e^{\theta_2}-e^{\theta_1} )|.
\end{equation}
Let us observe first that from \eqref{roest1}-\eqref{teta1} and the Mean Value Theorem we have
\begin{equation}\label{teta5}
e^\varphi |(e^{\rho(\eta_1)}-1)(e^{\theta_2}-e^{\theta_1} )|\leq \frac{C e^\theta  \rho(\eta_1)}{t} |\theta_2-\theta_1|\leq  \frac{(\ln t)^4}{t^3} |\eta_1-\eta_2|.
\end{equation}
To bound from above the first term in \eqref{teta4} we have
\begin{equation}\label{teta6}
e^{\varphi +\theta_2} |e^{\rho(\eta_2)}-e^{\rho(\eta_1)}|\leq\frac{ c }{t} |\rho(\eta_2)-\rho(\eta_1)|.
\end{equation}
For $\rho\in \Sigma_M$, from \eqref{ro0} we have
$$
\rho(\eta)=\eta^2 \int_0^1 A(z, \eta)dz\quad\mbox{ where }\; A(z, \eta)=G_m''(H_m(2t+\varphi)+ z\eta)(1-z).
$$
Then
$$
|\rho(\eta_2)-\rho(\eta_1)|\leq \int_0^1\Big[|A(z, \eta_2)||\eta_2+\eta_1||\eta_2-\eta_1|+|\eta_1|^2|A(z, \eta_2)-A(z, \eta_1)|\Big] dz.
$$
As before,
$$
|A(z, \eta_2)|\leq G_m''(H_m(2t+\varphi)+ z\eta)\leq G_m''(H_m(2t))\leq ct(\ln t)^3.
$$
Also, by the Mean Value Theorem, for some $\eta$ between $\eta_1$ and $\eta_2$ we have
$$
\begin{aligned}
|A(z, \eta_2)-A(z, \eta_1)|&=|\partial_\eta A(z, \eta)||\eta_2-\eta_1|\leq
G_m'''(H_m(2t))\leq ct(\ln t)^4.
\end{aligned}
$$
The last two estimates yield
\begin{equation}\label{teta7}
|\rho(\eta_2)-\rho(\eta_1)|\leq \frac{c(\ln t)^3}{t}|\eta_2-\eta_1|.
\end{equation}
Finally, combining \eqref{teta2}-\eqref{teta7} we deduce
$$
|F_2(t,\eta_2)-F_2(t,\eta_1)|\leq \frac{c(\ln t)^4}{t} |\eta_2-\eta_1|.
$$
Setting $\eta_2=0$ gives $F_2(t,\eta_1)\leq O(t^{-2})$ and we can appply the Lemma \ref{contraction}.
We then have a unique solution of \eqref{etam} in $\Sigma_M$.

Proceeding as in the proof of Theorem \ref{th1}, we have that there exists a $T^*$ such that
$w(T^*)=0$ and hence we have completed the proof of the existence
and uniqueness result.

We next derive the asymptotic expansion \eqref{asymptm}. We begin with
\begin{equation*}\begin{split}
w(t)=H_m(2t+\varphi)+O\left(\frac{1}{t^2}\right)=H_m(2t)+H'_m(2t)\varphi+\frac12 H''_m(2t)\varphi^2+ O\left(\frac{1}{t^2}\right).
\end{split}\end{equation*}
We have, using \eqref{estq2}, that
$$
H''_m(2t)=-\frac{H'_m(2t)}{2t}+O\left(\frac{1}{t^2(\ln t)^2}\right).
$$
Then
$$
\frac12 H''_m(2t)\varphi^2=-\frac{H'_m(2t)}{4t}\varphi^2+O\left(\frac{1}{t^2}\right)=-\frac{H'_m(2t)}{4t}(\ln t)^2+O\left(\frac{1}{t^2}\right).
$$
Finally, we use use the expansion \eqref{fifi} to deduce \eqref{asymptm}. We now prove the estimate for the gradient.
Exactly as in the proof of Theorem \ref{th1}, we have that $\eta_t=O(\frac{1}{t^2})$.
Then
$$
w_t=H'_m(2t+\varphi)(2+\varphi_t)+O\left(\frac{1}{t^2}\right)=2H'_m(2t+\varphi)+O\left(\frac{1}{t^2}\right),
$$
\noindent since $|\varphi_tH'_m(2t+\varphi)|\leq c t^{-2}$.
We use the expansion
$$
H'_m(2t+\varphi)=H'_m(2t)+ O\left(\frac{1}{t^2}\right),
$$
\noindent since $|H''_m(2t)\varphi|\leq c t^{-2}$.
This completes the proof of Theorem \ref{th3}, using that $|v_r|=|w_t \frac{dt}{dr}|$.

\bibliographystyle{amsplain}

\end{document}